\documentclass{amse-new}

\usepackage[colorlinks,linkcolor=blue,citecolor=blue]{hyperref}
\usepackage{latexsym, amssymb, amsmath, amscd, amsthm, mathrsfs, bbm, bm}
\usepackage{algorithm, algorithmicx, algpseudocode, multirow, listings, lscape, booktabs, manyfoot, textcomp}
\usepackage[all, knot]{xy}
\xyoption{arc}
\usepackage{anysize}\marginsize{22mm}{22mm}{25mm}{25mm}
\usepackage[figuresright]{rotating}

\numberwithin{equation}{section} 

\begin{document}

 \PageNum{1}
 \Volume{201x}{Sep.}{x}{x}
 \OnlineTime{August 15, 201x}
 \DOI{0000000000000000}
 \EditorNote{Received x x, 201x, accepted x x, 201x}

\abovedisplayskip 6pt plus 2pt minus 2pt \belowdisplayskip 6pt
plus 2pt minus 2pt
\def\vsp{\vspace{1mm}}
\def\th#1{\vspace{1mm}\noindent{\bf #1}\quad}
\def\proof{\vspace{1mm}\noindent{\it Proof}\quad}
\def\no{\nonumber}
\newenvironment{prof}[1][Proof]{\noindent\textit{#1}\quad }
{\hfill $\Box$\vspace{0.7mm}}
\def\q{\quad} \def\qq{\qquad}
\allowdisplaybreaks[4]


\def \k {\mathbbm{k}}
\def \End {\operatorname{End}}
\def \V {\operatorname{V}_{n,d}}
\def \U {\operatorname{U}_{n,d}}
\def \Ind {\operatorname{Ind}_{n,d}}
\def \dim {\operatorname{dim}}
\def \rk {\operatorname{Rank}}
\def \tr {\operatorname{trace}}
\def \Id {\operatorname{Id}}
\def \v {\operatorname{v}}
\def \GL {\operatorname{GL}}
\def \Im {\operatorname{Im}}
\def \Q {\mathbbm{Q}}
\def \R {\mathbbm{R}}
\def \C {\mathbbm{C}}
\def \Z {\mathbbm{Z}}

\numberwithin{equation}{section}
\numberwithin{table}{section}
\numberwithin{equation}{section}
\newtheorem{remarks}[theorem]{Remarks}
\newtheorem{convention}[theorem]{Convention}


\AuthorMark{Fang, Huang and Li}                             

\TitleMark{Complexity analysis of diagonalization}  

\title{Numerical algorithm and complexity analysis for diagonalization of \\ multivariate homogeneous polynomials
\footnote{Supported by Key Program of Natural Science Foundation of Fujian Province (Grant no. 2024J02018) and National Natural Science Foundation of China (Grant no. 12371037).}}                  

\author{Lishan \uppercase{Fang}}             
    {School of Mathematical Sciences, Huaqiao University, Quanzhou 362021, China \\
    E-mail\,$:$ fanglishan@hqu.edu.cn}

\author{Hua-Lin \uppercase{Huang}}     
    {School of Mathematical Sciences, Huaqiao University, Quanzhou 362021, China\\
    E-mail\,$:$ hualin.huang@hqu.edu.cn}

\author{Yuechen \uppercase{Li}}     
    {School of Mathematical Sciences, Huaqiao University, Quanzhou 362021, China\\
    E-mail\,$:$  liyuechen1225@163.com}
    
\maketitle%

\Abstract{We study the computational complexity of a diagonalization technique for multivariate homogeneous polynomials, that is, expressing them as sums of powers of independent linear forms. It is based on Harrison's center theory and consists of a criterion and a diagonalization algorithm. Detailed formulations and computational complexity of each component of the technique are given. The complexity analysis focuses on the impacts of the number of variables and the degree of given polynomials. We show that this criterion runs in polynomial time and the diagonalization process performs efficiently in numerical experiments. Other diagonalization techniques are reviewed and compared in terms of complexity.}      


\MRSubClass{15A20, 15A69}      

\section{Introduction}\label{sec:intro}

\subsection{Problem statement}\label{sec:problem}
In this article, we are concerned about the diagonalization of any complex multivariate homogeneous polynomial $f(x_1, x_2, \ldots, x_n) \in\C[x_1,x_2,\ldots,x_n]$ of degree $d \ge 3$. A multivariate homogeneous polynomial $f$ of degree $d$ is called diagonalizable if, after a linear change of variables $x=Py,$ it can be written as 
\[ 
    f(x)=f(Py)=\lambda_1 y_1^d + \lambda_2 y_2^d + \cdots + \lambda_n y_n^d, 
\] 
where $\lambda_i \in \C$. If this is possible, determine the appropriate change of variables $x=Py$. This is a typical problem in classical invariant theory and also attracts attention from theoretical computer science and scientific computing.

\subsection{Related work}\label{sec:related}

The diagonalization of homogeneous polynomials has been studied by various researchers, which covers both criteria for diagonalizability and algorithms for diagonalization. Kayal~\cite{kayal2011efficient} presented a randomized polynomial time algorithm based on polynomial factorization to determine diagonalizability and diagonalize homogeneous polynomials. Koiran~\cite{koiran2021derandomization} proposed a deterministic criterion for diagonalizability of homogeneous polynomials with $d=3$. Built on this criterion, a new randomized black-box criterion for $f$ with $d>3$ was given in~\cite{koiran2023absolute}.  
In contrast, Robeva~\cite{robeva2016orthogonal} addressed this problem from the perspective of symmetric tensors using the tensor power method. Some other studies are interested in the properties of diagonalization, such as orthogonality and unitarity~\cite{koiran2021orthogonal,kolda2000orthogonal,boralevi2017orthogonal,robeva2016orthogonal}.  

 

Harrison~\cite{harrison1975grothendieck} generalized Witt's algebraic theory of quadratic forms to higher degree forms in the 1970s, which was later applied to decompose higher degree forms~\cite{huang2021diagonalizable, huang2021centres}. Huang et al.~\cite{huang2021diagonalizable} presented a diagonalization technique based on Harrison's center theory and properties of sums of powers. It contains a criterion for diagonalizability and an iterative process to diagonalize $f$ with arbitrary $n$ and $d$, which only requires solving systems of linear and quadratic equations. In addition, it determines whether $f$ is orthogonally or unitarily diagonalized using Harrison's uniqueness of the decomposition.

\subsection{Methods}\label{sec:method}
The diagonalization techniques of homogeneous polynomials generally become computationally expensive as the dimensions or degrees of the polynomials increase.
Some~\cite{kayal2011efficient, robeva2016orthogonal} rely on tasks such as polynomial factorization or the tensor power method, which become the computational bottleneck of their techniques when given polynomials become complex~\cite{kaltofen1989factorization}.
Koiran and Skomra~\cite{koiran2021derandomization} provided the complexity analysis on a randomized and deterministic criterion, showing that their complexity is polynomially bounded. However, they may fail with a low probability as shown in their error analysis.

In this article, we provide a complexity analysis of the center-based diagonalization algorithm in~\cite{huang2021diagonalizable}. 
The criterion and diagonalization algorithm are divided into six steps, and their complexity is analyzed separately. Two examples are provided to illustrate the algorithm's performance. We also compare it with other techniques in terms of their computational complexity. Two numerical experiments are conducted to verify the theoretical analysis and results are provided in Appendices~\ref{app:cost} and~\ref{app:no_terms}.

\subsection{Organization of the paper}\label{sec:organize}

The remainder of this paper is organized as follows. We show Harrison's center theory and a diagonalization algorithm in Section~\ref{sec:center}. 
The criterion of diagonalizability is described in Section~\ref{sec:criteria}. The diagonalization process and the determination of orthogonality and unitarity are discussed in Section~\ref{sec:diagonalize}. We compare its computational complexity with other diagonalization techniques in Section~\ref{sec:compare}. Lastly, major findings are summarized in Section~\ref{sec:summary}. 

\section{Centers and diagonalization of homogeneous polynomials}\label{sec:center}

\subsection{Centers and diagonalization criterion}\label{sec:harrison_center}


We recall some necessary definitions of center algebras of homogeneous polynomials below. More details are given in~\cite{huang2021diagonalizable}. Let $f(x_1, x_2, \ldots, x_n) \in \C[x_1, x_2, \ldots, x_n]$ be a homogeneous polynomial of degree $d$. Its center~$Z(f)$ is defined as
\begin{equation}\label{eqn:center}
	Z(f):=\left\{ X \in \C^{n \times n} \mid (HX)^T=HX \right\},
\end{equation}
where $H=\left(\frac{\partial^2 f}{\partial x_i \partial x_j}\right)_{1 \le i, j \le n}$ is the Hessian matrix of $f$. 
    
\begin{remark} \label{sp}
Consider the sum of powers of variables $f(x_1, x_2, \ldots, x_n)=x_1^d+x_2^d+\cdots+x_n^d$. The Hessian matrix of $f$ is diagonal and its $ii$-entry is $d(d-1)x_i^{d-2}$. Thus, the $ij$-entry of $HX$ reads $d(d-1)x_i^{d-2}X_{ij}$ for $X=(X_{ij})_{1 \le i, j \le n}$. Since $HX$ is symmetric, we can see that $Z(f)$ consists of all diagonal matrices. Thus $Z(f)$ is a commutative associative subalgebra of the full matrix algebra and is isomorphic to $\C^n$ as algebras.
\end{remark}
	
A homogeneous polynomial $f$ is said to be nondegenerate if none of its variables can be removed by an invertible linear change of variables. In other words, $f$ is degenerate if and only if there exists a change of variables $x=Py$ such that 
    \[
        g\left(y_1, y_2,\ldots, y_n\right):=f(Py)=g\left(y_1, y_2, \ldots, y_{n-1}, 0\right),
    \]
which implies that $\frac{\partial g}{\partial y_n}=0.$ We also have $\frac{\partial g}{\partial y_n} = \sum_{i=1}^n \frac{\partial f}{\partial x_i} \frac{\partial x_i}{\partial y_n}$ by the chain rule of derivatives. Therefore, $f$ is degenerate if and only if its first-order differentials $\frac{\partial f}{\partial x_i}$ are linearly dependent. Additionally, any $f$ can be reduced to a nondegenerate one in a suitable number of variables. 

\begin{proposition}\label{prop:reduce}
    Let $f(x_1, x_2, \ldots, x_n)$ be as above and assume $\rk \left\{\frac{\partial f}{\partial x_1}, \frac{\partial f}{\partial x_2},\ldots, \frac{\partial f}{\partial x_n} \right\}=r.$ Then there exists a change of variables $x=Py$ such that
	\[ 
		g\left(y_1, y_2, \ldots, y_n\right):=f(Py)=g\left(y_1, y_2, \ldots, y_r, 0, \ldots, 0\right).
	\] 
Moreover, if we remove the redundant variables from $g(y)$ and view it as an $r$-variate polynomial, then it becomes nondegenerate. 
\end{proposition}

\begin{proof}
Without loss of generality, we may assume that $\left\{\frac{\partial f}{\partial x_1},\frac{\partial f}{\partial x_2}, \ldots,\frac{\partial f}{\partial x_r}\right\}$ are linearly independent and
	\[
	\frac{\partial f}{\partial x_{j}}=\sum_{i=1}^{r} \lambda_{j,i}\frac{\partial f}{\partial x_i} \hspace{1cm} \text{for} \hspace{0.11cm} j=r+1,\ldots,n,
	\]
	where~$\lambda_{j,i}\in \C$. Let $P$ take the form of
	\[
    	P =\left( \begin{array} {cccc|cccc}
    	1 & & & & \lambda_{r+1,1} & \lambda_{r+1,2} & \cdots & \lambda_{r+1,r} \\
    	& 1 & & & \lambda_{r+2,1} & \lambda_{r+2,2} & \cdots & \lambda_{r+2,r} \\
        & & \ddots &  & \vdots & \vdots &  & \vdots \\
    	& & & 1 & \lambda_{n,1} & \lambda_{n,2} & \cdots & \lambda_{n,r} \\
    	\hline
    	& & & & -1 & & & \\
        & & & & & -1 & & \\
        &  \multicolumn{2}{c}{\raisebox{2ex}[0pt]{\Huge0}} & & & & \ddots & \\
    	& & &  & & & & -1
    	\end{array} \right)
	\]
and take the change of variables $x=Py.$ As before, let $g(y)=f(Py).$ Then we have
 
    \begin{align*}
        \frac{\partial g}{\partial y_j}&=\frac{\partial f}{\partial x_j} &\text{for}\,\, 1 \le j \le r, \\
        \frac{\partial g}{\partial y_j} &=\sum_{i=1}^r \frac{\partial f}{\partial x_i}\frac{\partial x_i}{\partial y_j}+\sum_{i=r+1}^n \frac{\partial f}{\partial x_i}\frac{\partial x_i}{\partial y_j}=\sum_{i=1}^r \lambda_{j,i}\frac{\partial f}{\partial x_i}-\frac{\partial f}{\partial x_j}=0 &\text{for}\,\, j > r.
    \end{align*}
That is to say, variables $y_{r+1}, \ldots, y_n$ do not occur in $g(y)$ and we may view $g$ as an $r$-variate polynomial. In addition, $g$ is nondegenerate in $\C[y_1, y_2 \ldots, y_r]$ as $\left\{\frac{\partial g}{\partial y_1},\frac{\partial g}{\partial y_2}, \ldots, \frac{\partial g}{\partial y_r}\right\}$ are linearly independent.
\end{proof}
    
With the above proposition, in the following, there is no harm in considering only nondegenerate polynomials for our purpose. We can use the center in the criterion for determining whether a homogeneous polynomial $f$ is diagonalizable. The following proposition was essentially obtained in~\cite{harrison1975grothendieck}, see also~\cite{huang2021diagonalizable,huang2021centres,huang2022harrison}.
	
\begin{proposition} \label{cd}
    Suppose $f(x_1, x_2, \ldots, x_n) \in \C[x_1, x_2, \ldots, x_n]$ is a nondegenerate homogeneous polynomial of degree $d\ge 3.$ Then $f$ is diagonalizable if and only if $Z(f) \cong \C^n$ as algebras.
\end{proposition}
	
\begin{proof}
If $f$ is nondegenerate and diagonalizable, then there is a change of variables $x=Py$ such that
	\begin{equation*}
		g(y):=f(Py)=\lambda_1y_1^d+\lambda_2y_2^d+\cdots+\lambda_ny_n^d,
	\end{equation*}
where $\prod_{i=1}^n \lambda_i \ne 0.$ Let $G=\left(\frac{\partial^2g}{\partial y_i \partial y_j}\right)_{1 \le i, j \le n}$ denote the Hessian matrix of $g$ with respect to the variables $y_i.$ Then by the chain rule of derivatives, it is easy to see that $G=P^THP$. Thus, we have
    \[
        (GY)^T=GY \Leftrightarrow (HPYP^{-1})^T=HPYP^{-1}.
    \]
It follows that $Z(f) = PZ(g)P^{-1}.$ According to Remark~\ref{sp}, it is ready to see that $Z(g)$ consists of all diagonal matrices. Hence $Z(f)$ is isomorphic to $\C^n$ as algebras.

Conversely, if $Z(f) \cong \C^n$, then $Z(f)$ has a basis consisting of orthogonal primitive idempotents, say $\epsilon_1,$ $\epsilon_2,$ \ldots, $\epsilon_n.$ Since $\epsilon_1,$ $\epsilon_2,$ \ldots, $\epsilon_n$ mutually commute and are diagonalizable, there exists an invertible matrix $P$ such that $P^{-1}\epsilon_iP=E_i$ for all $i$, where $E_i$'s $ii$-entry is $1$ and 0 elsewhere. Take the change of variables $x=Py,$ then it is not hard to verify that $g(y)$ is diagonal. Indeed, by the above argument, we have $Z(g)=P^{-1}Z(f)P$ and $Z(g)$ consists of all diagonal matrices. Therefore, $G$ must be diagonal due to the condition of center algebras. This happens only if $g$ is diagonal and hence, $f$ is diagonalizable. This proves the statement.
\end{proof}

It is also of interest whether a polynomial $f$ is orthogonally, or unitarily diagonalizable, namely, whether there exists a change of variables $x=Py$ in which $P$ is orthogonal or unitary so that $f(Py)$ is diagonal \cite{robeva2016orthogonal,koiran2021orthogonal,kolda2000orthogonal,boralevi2017orthogonal}. The previous proposition can be applied to this problem effectively as Harrison states that the diagonalization of a diagonalizable polynomial is essentially unique \cite{harrison1975grothendieck}. This is also obvious from the fact that a finite-dimensional commutative algebra has a unique set of orthogonal primitive idempotents disregarding the order, see \cite{huang2023centers}. Thus, we can determine whether $f$ is orthogonally or unitarily diagonalizable by directly checking $P$.

\subsection{Diagonalization algorithm}\label{sec:diagonalization}

We show an algorithm to diagonalize given homogeneous polynomials in Algorithm~\ref{alg:diagonalization}. It consists of two phases: determining whether a polynomial~$f$ is diagonalizable from steps 1 to 3; and an iterative process for diagonalizing $f$ from steps 4 to 9. We divide this algorithm into six major steps, which are detect nondegeneracy, compute center, detect diagonalizability, compute idempotents, diagonalize form, and detect orthogonality, corresponding to steps 1, 2, 3, 6, 7, and 9, respectively. Detailed descriptions, complexity analysis, and examples are given in Sections~\ref{sec:criteria} and~\ref{sec:diagonalize}. 
		
	\begin{algorithm}
		\caption{Diagonalization of homogeneous multivariate polynomial}\label{alg:diagonalization}
		\begin{algorithmic} [1]
		\Statex \textbf{Input}: polynomial~$f$
		\Statex \textbf{Output}: diagonal form of $f$
		\State Determine linear independence of $\left\{\frac{\partial f}{\partial x_1}, \frac{\partial f}{\partial x_2}, \ldots, \frac{\partial f}{\partial x_n}\right\}$. If they are linearly independent, $f$ is nondegenerate; otherwise, transform $f$ to a nondegenerate form in fewer variables according to Proposition~\ref{prop:reduce}.
		\State Solve system $(HX)^T=HX$ to compute center~$Z(f)$. If $\dim Z(f) \neq n$, reject as $f$ is not diagonalizable.
		\State Use the Euclid algorithm to determine whether the minimal polynomials of basis matrices $X_1,X_2,\ldots, X_n$ of $Z(f)$ have multiple roots. If any multiple roots are found, reject as $f$ is not diagonalizable.
		\While {$f$ is not in a diagonal form yet}
		\State Compute center~$Z(f)$ if required.
		\State Compute orthogonal idempotents~$(\epsilon, I-\epsilon)$ using $Z(f)$ with $\rk \epsilon=1$ and $\tr(\epsilon)=1$ conditions.
		\State Decompose the form of~$f$ using change of variables $x=Py$ and separate disjoint variables.
		\EndWhile
		\State Determine orthogonality and unitarity of diagonalization based on $P$.
		\end{algorithmic}
	\end{algorithm}
	
The complexity analysis in this article focuses on the theoretical upper bounds of the computational costs in terms of three factors: dimension~$n$, degree~$d$, and the number of monomial terms~$t$ of~$f$. As $n$, $d$, and $t$ increase, the costs of the diagonalization algorithm rise. Values of~$n$ and~$d$ are independent in $f$ and $t$ increases if either $n$ or $d$ increases with an upper bound $\binom{n+d-1}{d}$. While the values of~$t$ increase exponentially for large~$n$ and~$d$, the growth is only polynomial if either~$n$ or~$d$ is fixed. In addition, $t$ does not affect the dimensions of the two systems of equations in Algorithm~\ref{alg:diagonalization}. 
Discussions of the growth of $t$ are provided in Appendix~\ref{app:no_terms}.

\section{Criterion of diagonalizability}\label{sec:criteria}
The criterion for determining the diagonalizability of given~$f$ is described in three steps: detect nondegeneracy, compute center, detect diagonalizability. We provide their formulations and complexity analysis in Subsections~\ref{sec:nondegeneracy},~\ref{sec:compute_center} and~\ref{sec:diagonal}, respectively. It is shown that this criterion has polynomial runtime. In addition, two examples are given to demonstrate its procedure in Subsection~\ref{sec:c_example}.

\subsection{Detect nondegeneracy}\label{sec:nondegeneracy}
	
The process begins by detecting and reducing the nondegeneracy of $f$ before the diagonalization as shown in Algorithm~\ref{alg:nondegeneracy}. Recall that $f$ is nondegenerate if and only if its first-order differentials $\frac{\partial f}{\partial x_i}$ are linearly independent as discussed in Proposition~\ref{prop:reduce}. The linear dependence of $\frac{\partial f}{\partial x_i}$ for $i=1,\ldots,n$ is determined by transforming them to an $n\times t_{max}$ matrix $A$, where its $ij$-th entry $A_{ij}$ equals to the coefficient of $\frac{\partial f_{(j)}}{\partial x_i}$, $f_{(j)}$ is the $j$-th term of $f$ and $t_{max}=\binom{n+d-1}{d}$. Here and below, the terms of monomials of $f$ are lexicographically ordered. We then have
\[
    \rk A = \rk \left\{\frac{\partial f}{\partial x_1}, \frac{\partial f}{\partial x_2}, \cdots,\frac{\partial f}{\partial x_n}\right\}.
\]
If $\rk A=n$, then $\frac{\partial f}{\partial x_i}$ are linearly independent and $f$ is nondegenerate. Otherwise, $f$ is degenerate and we reduce~$f$ into a nondegenerate form in fewer variables.


\begin{algorithm}
		\caption{Detect and reduce nondegeneracy of polynomial}\label{alg:nondegeneracy}
		\begin{algorithmic} [1]
		\Statex \textbf{Input}: polynomial $f$
		\Statex \textbf{Output}: nondegenerate form of $f$
		\State Calculate~$\frac{\partial f}{\partial x_i}$ for~$i=1,\ldots,n$ and build matrix $A$.
		\If {$\rk A<n$}
		\State Compute RREF of $A$ and obtain $r$ maximum linearly independent subset $\left\{\frac{\partial f}{\partial x_1},\frac{\partial f}{\partial x_2},\ldots,\frac{\partial f}{\partial x_r}\right\}$.
		\State Reduce~$f$ into a nondegenerate form with $r$ variables.
		\EndIf
		\end{algorithmic}
\end{algorithm}
	
We compute the reduced row echelon form (RREF) of $A$ using Gaussian elimination to identify the maximum linearly independent subset $\left\{\frac{\partial f}{\partial x_1},\frac{\partial f}{\partial x_2},\ldots,\frac{\partial f}{\partial x_r}\right\}$ with rank $r<n$. The other dependent differentials can be expressed as
	\[
	\frac{\partial f}{\partial x_{j}}=\sum_{i=1}^{r} \lambda_{j,i}\frac{\partial f}{\partial x_i} \hspace{1cm} \text{for} \hspace{0.11cm} j=r+1,\ldots,n.
	\]
We obtain an invertible matrix~$P$ as in Proposition~\ref{prop:reduce} and a nondegenerate polynomial $g=f(Py)$ in $r$ variables. Note that we replace $g$ and $r$ with $f$ and $n$, respectively, in the rest of the article for the nondegenerate form and the original $n$ is represented as $\bar{n}$.

The construction of $A$ involves the calculation of derivatives $\frac{\partial f}{\partial x_i}$, which takes $O(n^2t)$ arithmetic operations. $\rk A$ is calculated by computing the RREF of $A$ using Gaussian elimination with $O(n^3)$ arithmetic operations~\cite{golub2013matrix}. While faster algorithms for rank calculation like~\cite{cheng2013fast} achieve $O(nt)$ complexity for $A$ with $n<t$, the RREF of $A$ is also required for reducing the degeneracy and prefered. Therefore, the complexity of detecting the nondegeneracy of $f$ is $O(n^2t)$.

Reducing the degeneracy of $f$ takes $O(nd)$ operations to reduce one variable by a change of variables. It also needs to expand at most $d$ terms of~$f$ in this process, which involves numerous multivariate polynomial multiplications. Since multiplying two polynomials takes $O(n\log{n}\log{\log{n}})$ operations for polynomials with $d<n$, this process takes at most $O(n^2d\log{n}\log{\log{n}})$ operations~\cite{cantor1991fast}. Thus, we deduce that the overall computational complexity of the detect nondegeneracy step is $O(n^2t)$. 
	
\begin{proposition} \label{pro:nondegeneracy}
Detecting and reducing the nondegeneracy of $f$ requires $O(n^2t)$ arithmetic operations.
\end{proposition}

\subsection{Compute center}\label{sec:compute_center}
We continue to compute center $Z(f)$ of a nondegenerate form $f$, which is the core of this diagonalization algorithm. This step as described in Algorithm~\ref{alg:center} is called at most $n-1$ times during the iterative process.

\begin{algorithm}
	\caption{Compute center of polynomial}\label{alg:center}
	\begin{algorithmic} [1]
	\Statex \textbf{Input}: polynomial $f$
	\Statex \textbf{Output}: Center $Z(f)$ and basis matrices $X_{1},X_{2}, \ldots,X_{n}$
	\State Compute Hessian matrix~$H$ of $f$ and build a system of equations~$CX_v=0$ equivalent to $(HX)^T=HX$.
	\State Compute the RREF of~$C$ and obtain $Z(f)$. If $\dim Z(f)\neq n$, reject as $f$ is not diagonalizable.
	\State Obtain basis matrices~$X_{1},X_{2},\ldots,X_{n}$ from $Z(f)$.
	\end{algorithmic}
\end{algorithm}
		
Recall that center $Z(f)$ is computed by solving the matrix equation $(HX)^T=HX$ in Equation~\eqref{eqn:center}, which can be transformed into a system of linear equations. 
Given an $n\times n$ matrix~$X=\left(x_{ij}\right)_{1 \le i, j \le n}$, let~$X_v$ be the $n^2$-dimensional column vector
	\[
	\begin{pmatrix}
	X_{(1)} \\
	\vdots \\
	X_{(n)}
	\end{pmatrix},
	\]
where $X_{(i)}$ is the $i$-th column of~$X$. Let $X_v^T$ denote the corresponding vector of $X^T$ and $P_{C}$ be the permutation matrix such that $X_v^T=P_{C}X_{v}$. We denote $M_{d-2}$ as the set of monomials of degree $d-2$ and express the Hessian matrix $H$ as $H=\sum_{m \in M_{d-2}} mH_m.$ Then $(HX)^T=HX$ can be rewritten as a system of linear equations 
\[ 
    CX_v = \left(I_n \otimes H_m - \left( H_m \otimes I_n \right) P_C \right) X_v=0, \quad m \in M_{d-2}, 
\]
where $\otimes$ is the tensor product for multilinear algebra, see \cite{huang2021centres} for more details. 
We obtain $Z(f)$ by solving $CX_v = 0$ using Gaussian elimination and computing the RREF $C_{R}$. The resulting column vector $X_v$ is transformed back to $n\times n$ matrix~$X$. If $\dim Z(f) = n$, basis matrices~$X_{1},X_{2},\ldots, X_{n}$ of~$Z(f)$ are chosen for each free variable in $C_{R}$ by identifying the linear dependence between the variables. If $\dim Z(f) \neq n$, we deduce that $f$ is not diagonalizable.

A significant portion of the computational cost in this step involves constructing matrix~$C$ and calculating $C_{R}$. We build $C$ by iterating through $\left(I_n \otimes H_m - \left( H_m \otimes I_n \right) P_C \right)_{m \in M_{d-2}}$ and the process takes at most $O(n^3d^3t)$ arithmetic operations. Matrix $C$ is then reduced to $C_{R}$ using Gaussian elimination, which requires $O(n^{d+4})$ operations~\cite{boyd2018introduction}. Thus, it is polynomial-time if either $n$ or $d$ is fixed. Additionally, while $C$ is of size at most $n^d \times n^2$ in theory~\cite{huang2021centres}, it contains a large number of redundant equations and zero entries, which help to reduce the costs of Gaussian elimination. For instance, $C$ in Example~\ref{sec:c_example} has dimension~$261\times16$ after removing redundant rows, which is markedly smaller than the theoretical upper bound~$1024\times 16$. More statistics are given in Table~\ref{tab:stat_n}, which shows that the dimension of $C$ increases with order $O(n^{4.50})$. Appendix~\ref{app:cost} also demonstrates that the time for solving $CX_v=0$ increases with order~$O(n^{6.66})$ for $n\le7$. Lastly, the costs of building $H$ and choosing $X_1,X_2,\ldots, X_n $ are trivial and the overall complexity of computing center $Z(f)$ is $O(n^3d^3t+n^{d+4})$.
 
\begin{proposition} \label{pro:compute_center}
    Computing center~$Z(f)$ takes at most $O(n^3d^3t+n^{d+4})$ arithmetic operations.
\end{proposition}

\subsection{Detect diagonalizability}\label{sec:diagonal}
The diagonalizability of~$f$ is determined using semisimplicity of $Z(f)$'s basis matrices~$X_{1},X_2,\ldots,X_{n}$ as shown in Algorithm~\ref{alg:diagonal}. Minimal polynomial~$p_i(\lambda)$ for each~$X_{i}$ is built by iteratively testing the linear dependence of its powers $X_i^1,X_i^2,\ldots, X_i^j$ for $j \le n$. Let matrix~$D$ of size $n^2 \times (j+1)$ take the form of
\[
    D = \begin{pmatrix} I_{n,v} & D_{1,v} & \cdots & D_{j,v} \end{pmatrix},
\]
where $I_{n,v}$ and $D_{j,v}$ are $n^2$-dimensional column vectors of matrices $I_n$ and $X_{i}^{j}$, respectively. They are defined similarly to $X_v$ of $X$ in Section~\ref{sec:compute_center}. This process takes at most $n-1$ times of matrix multiplications and at most $O(n^4)$ operations in total. 
If $\rk D < j$, $X_i^1, X_i^2, \ldots,X_i^j$ are linearly dependent. The corresponding minimal polynomial $p_i$ is represented as
\[
    p_i(\lambda) = \sum_{j=0}^{m_{i} } t_j\lambda^j,
\]
where $t_j$ is the leading coefficient of $j$-th row of the RREF of $D$. Compute the greatest common divisor (GCD) of $p_i(\lambda)$ and $p_i^{'}(\lambda)$ using the Euclid algorithm, which takes $O(n^2)$ operations~\cite{bini2012polynomial}. If $\text{GCD}\left(p_i(\lambda),p_i^{'}(\lambda)\right)=1$, $p_i(\lambda)$ does not have multiple roots. If the GCD of all basis matrices~$X_i$ is 1, we deduce that $f$ is diagonalizable.
	

\begin{algorithm}
		\caption{Detect diagonalizability of polynomial}\label{alg:diagonal}
		\begin{algorithmic} [1]
		\Statex \textbf{Input}: basis matrices $X_{1},X_{2},\ldots,X_{n}$
		\Statex \textbf{Output}: diagonalizability of $f$
		\For {each basis matrix~$X_i$ of $Z(f)$}
		\For {$j=1,\ldots,n$}
		\State Build matrix $D$ containing coefficients of $X_i^1, X_i^2, \ldots,X_i^j$. 
		\If {$X_i^1,X_i^2,\ldots,X_i^j$ are linearly dependent}

		\State Build minimal polynomial $p_i(\lambda)$. If $GCD\left(p_i(\lambda),p_i^{'}(\lambda)\right)\ne1$, $p_i(\lambda)$, $f$ is not diagonalizable.
		\EndIf
		\EndFor
		\EndFor
		\end{algorithmic}
\end{algorithm}
	
Algorithm~\ref{alg:diagonal} contains a nested loop for each basis matrix~$X_i$, which computes linear dependence for at most $n$-th power of $X_i$. Each inner loop consists of a matrix multiplication, a rank calculation, and the Euclid algorithm. Therefore, the overall computational complexity for determining diagonalizability is $O(n^5)$.
	
\begin{proposition} \label{pro:diagonal}
    Detecting the diagonalizability of $f$ has computational complexity of $O(n^5)$.
\end{proposition}
	
In summary, the criterion for diagonalizability involves the detect nondegeneracy, compute center, and detect diagonalizability steps with complexity $O(n^2t)$, $O(n^3d^3t+n^{d+4})$, and $O(n^5)$, respectively. The computation of center~$Z(f)$ is the most computationally expensive part of the criterion and the overall complexity of the criterion is $O(n^3d^3t+n^{d+4})$.
	
	

\subsection{Some examples}\label{sec:c_example}

\begin{example}\label{d_c1}
Consider the following quadruple quintic form
\begin{align*}	f&=-1023x_1^5+2720x_1^4x_2-5200x_1^4x_3-2725x_1^4x_4-2240x_1^3x_2^2+9600x_1^3x_2x_3+3840x_1^3x_2x_4-10000x_1^3x_3^2 \\
    &-9280x_1^3x_3x_4-1590x_1^3x_4^2 +1600x_1^2x_2^3-8640x_1^2x_2^2x_3-5760x_1^2x_2^2x_4+16320x_1^2x_2x_3^2+19200x_1^2x_2x_3x_4 \\
    &+7680x_1^2x_2x_4^2-10520x_1^2x_3^3-17040x_1^2x_3^2x_4-11040x_1^2x_3x_4^2-3530x_1^2x_4^3-160x_1x_2^4+1920x_1x_2^3x_3 \\
    &-6720x_1x_2^2x_3^2-3840x_1x_2^2x_3x_4+1920x_1x_2^2x_4^2+9600x_1x_2x_3^3+11520x_1x_2x_3^2x_4-3840x_1x_2x_4^3\\
    &-4970x_1x_3^4-9040x_1x_3^3x_4-4080x_1x_3^2x_4^2+2240x_1x_3x_4^3+2085x_1x_4^4+64x_2^5-480x_2^4x_3-480x_2^4x_4\\
    &+1600x_2^3x_3^2+2560x_2^3x_3x_4+1600x_2^3x_4^2-2880x_2^2x_3^3-5760x_2^2x_3^2x_4-5760x_2^2x_3x_4^2-2880x_2^2x_4^3+2720x_2x_3^4\\
    &+6400x_2x_3^3x_4+7680x_2x_3^2x_4^2+6400x_2x_3x_4^3+2720x_2x_4^4-1055x_3^5-2870x_3^4x_4-3800x_3^3x_4^2-3760x_3^2x_4^3\\
    &-2800x_3x_4^4-1025x_4^5
\end{align*}
in $\C[x_1,x_2,x_3,x_4]$. Calculate its partial derivatives $\frac{\partial f}{\partial x_1} ,\ldots, \frac{\partial f}{\partial x_4}$ and build matrix $A$. The calculation gives $\rk A=4$, which shows that $\frac{\partial f}{\partial x_i}$ are linearly independent and $f$ is nondegenerate. Compute $Z(f)$ according to Equation~\eqref{eqn:center} and a general solution reads $X=\lambda_{1}X_1+\lambda_{2}X_2+\lambda_{3}X_3+\lambda_{4}X_4$, where
\begin{align*}
	&X_1 = \begin{pmatrix} 8 & 0 & 0 & -3 \\ 3 & 5 & 0 & -3 \\ -4 & 0 & 5 & 4 \\ 5 & 0 & 0 & 0 \end{pmatrix}, \quad
    X_2=\begin{pmatrix} 3 & 5 & 0 & -3 \\ 28 & 20 & -20 & -43 \\ 16 & 0 & -5 & -16 \\ 0 & 5 & 0 & 0 \end{pmatrix}, \\
	&X_3 =\begin{pmatrix} -4 & 0 & 0 & 4 \\ 16 & 0 & -5 & -16 \\ 12 & 0 & -10 & -12 \\ 5 & 0 & 0 & 0 \end{pmatrix}, \quad X_4=\begin{pmatrix} 1 & 0 & 0 & 0 \\ 0 & 1 & 0 & 0 \\ 0 & 0 & 1 & 0 \\ 0 & 0 & 0 & 1 \end{pmatrix}.
\end{align*}
The minimal polynomials~$p_i(\lambda)$ of $X_1,\ldots,X_4$ are $\lambda^2-1.6\lambda+0.6$, $\lambda^4-3.6\lambda^3+0.8\lambda^2+3.6\lambda-1.8$, $\lambda^3+2.8\lambda^2+1.6\lambda$ and $\lambda-1$, respectively. Since $GCD\left(p_i(\lambda),p_i^{'}(\lambda)\right)=1$ for $i=1,\ldots,4$, we deduce that $f$ is diagonalizable.
	
\end{example}
	
\begin{example}\label{d_c2}
Consider the following sextuple septic form
\begin{align*}	
    f&=128x_1^7+448x_1^6x_2-7x_1^6x_3+14x_1^6x_5+672x_1^5x_2^2-21x_1^5x_3^2+84x_1^5x_5^2+560x_1^4x_2^3-35x_1^4x_3^3 +280x_1^4x_5^3\\
    &+280x_1^3x_2^4-35x_1^3x_3^4+560x_1^3x_5^4+84x_1^2x_2^5-21x_1^2x_3^5+672x_1^2x_5^5+14x_1x_2^6-7x_1x_3^6 +448x_1x_5^6+x_2^7\\
    &-7x_3^6x_4-21x_3^6x_6 +21x_3^5x_4^2+126x_3^5x_4x_6+189x_3^5x_6^2-35x_3^4x_4^3-315x_3^4x_4^2x_6 -945x_3^4x_4x_6^2-945x_3^4x_6^3\\
    &+35x_3^3x_4^4+420x_3^3x_4^3x_6+1890x_3^3x_4^2x_6^2+3780x_3^3x_4x_6^3+2835x_3^3x_6^4 -21x_3^2x_4^5-315x_3^2x_4^4x_6-1890x_3^2x_4^3x_6^2\\
    &-5670x_3^2x_4^2x_6^3-8505x_3^2x_4x_6^4-5103x_3^2x_6^5+7x_3x_4^6 + 126x_3x_4^5x_6+945x_3x_4^4x_6^2+3780x_3x_4^3x_6^3\\
    &+8505x_3x_4^2x_6^4+10206x_3x_4x_6^5+5103x_3x_6^6-2x_4^7+14x_4^6x_5-21x_4^6x_6-84x_4^5x_5^2-189x_4^5x_6^2+280x_4^4x_5^3\\
    &-945x_4^4x_6^3-560x_4^3x_5^4-2835x_4^3x_6^4+672x_4^2x_5^5-5103x_4^2x_6^5-448x_4x_5^6-5103x_4x_6^6+2443x_5^7\\
    &+10206x_5^6x_6+20412x_5^5x_6^2+22680x_5^4x_6^3+15120x_5^3x_6^4+6048x_5^2x_6^5 +1344x_5x_6^6-2059x_6^7
\end{align*}
in $\C[x_1,x_2,x_3,x_4,x_5,x_6]$. Similar to the previous example, build matrix $A$, which shows $\rk A=6$. This shows that $\frac{\partial f}{\partial x_i}$ are linearly independent and $f$ is nondegenerate. By a direct calculation, the general solution of $Z(f)$ reads $X=\lambda_{1}X_1+\lambda_{2}X_2+\lambda_{3}X_3+\lambda_{4}X_4+\lambda_{5}X_5+\lambda_{6}X_6$, where
\begin{equation*}
	X_1 = \begin{pmatrix} 0 & 0 & 0 & 0 & 0 & 0 \\ 0 & 2 & 1 & 0 & 0 & 0 \\0 & 0 & 0 & 0 & 0 & 0 \\0 & 0 & 0 & 0 & 0 & 0 \\0 & 0 & 0 & 0 & 0 & 0 \\0 & 0 & 0 & 0 & 0 & 0 \end{pmatrix}, \, X_2 = \begin{pmatrix} 0 & 0 & 0 & 0 & -6 & -4 \\ 0 & 0 & 0 & 0 & 12 & 8 \\ 0 & 0 & 0 & 0 & 6 & 4 \\0 & 0 & 0 & 0 & 6 & 4 \\0 & 0 & 0 & 0 & 3 & 2 \\0 & 0 & 0 & 0 & 0 & 0 \end{pmatrix}, \, X_3 = \begin{pmatrix} 4 & 0 & 4 & 0 & 0 & 0 \\ -8 & 0 & -8 & 0 & 0 & 0 \\ 5 & 0 & 5 & 0 & 0 & 0 \\ -4 & 0 & -4 & 0 & 0 & 0 \\ -2 & 0 & -2 & 0 & 0 & 0 \\ 3 & 0 & 3 & 0 & 0 & 0 \end{pmatrix},
\end{equation*}
\begin{equation*}
	X_4 = \begin{pmatrix} -5 & 0 & 0 & 4 & -18 & 0 \\ 10 & 0 & 0 & -8 & 36 & 0 \\5 & 0 & 0 & -4 & 18 & 0 \\ -4 & 0 & 0 & -13 & 18 & 0 \\ -2 & 0 & 0 & -2 & 0 & 0 \\ 3 & 0 & 0 & 3 & 0 & 0 \end{pmatrix}, \, X_5 = \begin{pmatrix} -5 & 0 & 0 & 0 & -10 & 0 \\ 10 & 0 & 0 & 0 & 20 & 0 \\ 5 & 0 & 0 & 0 & 10 & 0 \\ -4 & 0 & 0 & 0 & -8 & 0 \\ -2 & 0 & 0 & 0 & -4 & 0 \\ 3 & 0 & 0 & 0 & 6 & 0 \end{pmatrix}, \, X_6 = \begin{pmatrix} 1 & 0 & 0 & 0 & 0 & 0 \\ -2 & 0 & 0 & 0 & 0 & 0 \\ 0 & 0 & 1 & 0 & 0 & 0 \\0 & 0 & 0 & 1 & 0 & 0 \\0 & 0 & 0 & 0 & 1 & 0 \\0 & 0 & 0 & 0 & 0 & 1 \end{pmatrix}.
\end{equation*}
The minimal polynomials~$p_i(\lambda)$ of $X_1,\ldots,X_6$ are $\lambda^2-\lambda$, $\lambda^2-1.5\lambda$, $\lambda^2-3\lambda$, $\lambda^2+3\lambda$, $\lambda^2+1.5\lambda$ and $\lambda^2-\lambda$, respectively. Since $GCD\left(p_i(\lambda),p_i^{'}(\lambda)\right)=1$ for $i=1,\ldots,6$, we deduce that $f$ is diagonalizable.
\end{example}

\section{Diagonalizing homogeneous polynomials}\label{sec:diagonalize}

After determining that~$f$ is diagonalizable, Algorithm~\ref{alg:diagonalization} iteratively computes orthogonal idempotents $(\epsilon,I-\epsilon)$ and decomposes $f$ until obtaining a diagonal form. The algorithm then determines whether the diagonalization is orthogonal and unitary. Our analysis shows that this process may not have polynomial time complexity since it involves solving a nonlinear system, which is NP-hard. However, numerical results indicate that its computation is performed efficiently. This process is analyzed in three steps: compute idempotents, diagonalize form, and detect orthogonality, which are discussed in Subsections~\ref{sec:idempotent},~\ref{sec:diagonal_form} and~\ref{sec:ortho_unit}, respectively. Two examples are given to demonstrate the iterative process of the diagonalization method in Subsection~\ref{sec:d_example}.

\subsection{Compute idempotents}\label{sec:idempotent}
	
A pair of orthogonal idempotents $(\epsilon,I-\epsilon)$ is computed by imposing $\rk \epsilon=1$ and $\tr(\epsilon)=1$ conditions on~$Z(f)$.
The resulting system of equations $N$ consists of both linear and quadratic equations. Note that we have obtained the RREF~$C_{R}$ for $Z(f)$ in previous steps, which are transformed into $n^2-n$ linear equations. 
The $\rk \epsilon=1$ condition is reinforced by ensuring rows of matrix $X=\left(x_{ij}\right)_{1 \le i, j \le n}$ are linearly dependent, such that
\[
    \lambda_i\left(x_{11},x_{12},\ldots,x_{1n})=(x_{i1},x_{i2},\ldots,x_{in}\right) \hspace{0.6cm} \text{for} \hspace{0.15cm} i=2,\ldots,n.	
\]
They are transformed into $(n-1)^2$ quadratic equations in the form of
\[
    x_{1,j} x_{i+1,j+1} = x_{1,j+1} x_{i+1,j} \hspace{0.6cm} \text{for} \hspace{0.15cm} i,j=1,\ldots,n-1.
\]
The $\tr(\epsilon)=1$ condition is imposed directly as $x_{11}+x_{22}+\cdots+x_{nn}=1$. The construction of these three sets of equations take~$O(n^3)$, $O(n^2)$ and~$O(n)$ operations, respectively. Thus, the overall construction of $N$ requires $O(n^3)$ operations.
	
    
System~$N$ contains $n^2-n+1$ linear equations and $n^2-2n+1$ quadratic equations. Note that $C_{R}$ is naturally sparse and each quadratic equation comprises two terms with four variables regardless of the size of~$n$ and~$d$. In addition, $X_{1},X_{2},\ldots,X_{n}$ of~$Z(f)$ often contain zero entries. Thus, even though the size of $N$ increases quadratically as~$n$ increases, $N$ remains relatively sparse and efficient to solve. System $N$ may have multiple solutions and one will be chosen as $\epsilon$ to decompose $f$. If no solution is found, we deduce that~$f$ is not diagonalizable.
	
Nonlinear systems are often solved numerically using iterative methods like Newton's method or its variants~\cite{cordero2012increasing}. Shin et al.~\cite{shin2010comparison} compared the Newton–Krylov method with other Newton-like methods and argued that its performance is superior given sparse and large systems. Nevertheless, solving random quadratic systems of equations is still known to be NP-hard in general~\cite{chen2017solving}. Though the numerical experiments in Appendix~\ref{app:cost} show that its computational costs increase with order~$O(n^{1.85})$ for $n\le7$. We denote the complexity of solving $N$ as $O(\rho)$ for the rest of the article and the complexity of this step is~$O(n^3+\rho)$. 
 
 


	\begin{proposition} \label{pro:idempotent}
		Computing an orthogonal idempotent~$\epsilon$ takes~$O(n^3+\rho)$ arithmetic operations.
	\end{proposition}

\subsection{Diagonalize form}\label{sec:diagonal_form}

The diagonalization algorithm decomposes $f$ using the orthogonal idempotents~$(\epsilon,I-\epsilon)$ as described in Algorithm~\ref{alg:diag_form}. If $f$ is not in a diagonal form after the decomposition, Algorithm~\ref{alg:diagonalization} will continue to compute~$Z(f)$ for the decomposed form of $f$ and calculate a new~$\epsilon$. This iterative process terminates until the form of $f$ becomes diagonal. In this section, we describe the diagonalize form step in an iterative setting and denote the idempotent in iteration $i$ as $\epsilon_i$.

	\begin{algorithm}
	\caption{Diagonalize polynomial~$f$}\label{alg:diag_form}
		\begin{algorithmic} [1]
		\Statex \textbf{Input}: idempotent $\epsilon$
		\Statex \textbf{Output}: decomposed form of $f$
		\State Compute permutation matrix~$P$ using a pair of orthogonal idempotents ($\epsilon,I-\epsilon$) 
        \State Apply change of variables $x=Py$ to decompose $f$
		\State If $f$ consists of a disjoint set of variables, it is fully diagonalized. Otherwise; replace $f$ by $g$ with $n-1$ variables and continue to diagonalize $f$.
		\end{algorithmic}
	\end{algorithm}
	
We first calculate a pair of orthogonal idempotents ($\epsilon_1$, $I-\epsilon_1$) , where
	\[
		\epsilon_1 = \begin{pmatrix}
		x_{11}^{(1)} & x_{12}^{(1)} & \cdots & x_{1n}^{(1)}\\
		x_{21}^{(1)} & x_{22}^{(1)} & \cdots & x_{2n}^{(1)}\\
		\vdots & \vdots &\ddots &\vdots \\
		x_{n1}^{(1)} & x_{n2}^{(1)} & \cdots & x_{nn}^{(1)}\\
		\end{pmatrix} \quad\text{and}\quad I-\epsilon_1=\begin{pmatrix}
		1-x_{11}^{(1)} & -x_{12}^{(1)} & \cdots & -x_{1n}^{(1)}\\
		-x_{21}^{(1)} &1-x_{22}^{(1)} & \cdots & -x_{2n}^{(1)}\\
		\vdots & \vdots &\ddots &\vdots \\
		-x_{n1}^{(1)} & -x_{n2}^{(1)} & \cdots & 1-x_{nn}^{(1)} \end{pmatrix}.
	\]
The permutation matrix~$P_1$ of~$f$ now takes the form of a combination of $\epsilon$ and $I-\epsilon$ as explained in Proposition 3.7 in~\cite{huang2021diagonalizable}, where
	\[
		P_1=\begin{pmatrix}
		x_{11}^{(1)} & x_{12}^{(1)} & \cdots & x_{1r}^{(1)}\\
		-x_{21}^{(1)} &1-x_{22}^{(1)} & \cdots & -x_{2n}^{(1)}\\
		\vdots & \vdots &\ddots &\vdots \\
		-x_{n1}^{(1)} & -x_{n2}^{(1)} & \cdots & 1-x_{nn}^{(1)}
		\end{pmatrix}.
	\]
Decompose~$f$ by applying a change of variables
\[
	\left(y_{1}^{(1)},y_{2}^{(1)},\ldots,y_{n}^{(1)}\right)^T=P_1 \left(y_1,y_2,\ldots,y_n\right)^T,
\]
where~$y_{1}^{(1)},y_{2}^{(1)},\ldots,y_{n}^{(1)}$ are new variables in the $1$-st iteration. 
We then obtain
\[
	f=c_1\left(y_{1}^{(1)}\right)^d+g_1\left(y_{2}^{(1)},y_{3}^{(1)},\ldots,y_{n}^{(1)}\right),
\]
where $c_1$ is a $d$-th power linear form and $g_1$ is a nondegenerate form of degree $d$ in $n-1$ variables.

	

If $g_1$ is diagonal, the diagonalization process completes and we continue to determine properties of the diagonalization; otherwise, we return to step 5 of Algorithm~\ref{alg:diagonalization}. After computing new center~$Z(g_1)$ and corresponding orthogonal idempotent~$(\epsilon_2,I-\epsilon_2)$, we have
	\[
		P_2 = \begin{pmatrix}
		I_1 & 0 & 0 & \cdots & 0\\
		0 &x_{22}^{(2)} & x_{23}^{(2)}&\cdots & x_{2r}^{(2)}\\
		0 &-x_{32}^{(2)} & 1-xp_{33}^{(2)}&\cdots & -x_{3n}^{(2)}\\
		\vdots & \vdots &\vdots &\ddots &\vdots \\
		0 & -x_{n2}^{(2)} & -x_{n3}^{(2)} & \cdots & 1-x_{nn}^{(2)}
		\end{pmatrix}.
	\]
Then decompose~$g_1$ by applying the change of variables
\[
    \left(y_{2}^{(2)},\ldots,y_{n}^{(2)}\right)^T=P_2 \left(y_{2}^{(1)},\ldots,y_{n}^{(1)}\right)^T
\]
and we have
\[
    g_1=c_2\left(y_{2}^{(2)}\right)^d+g_2\left(y_{3}^{(2)},\ldots,y_{n}^{(2)}\right).
\]
$g_2$ is a nondegenerate form of degree~$d$ in~$n-2$ variables and is independent from~$y_{1}^{(1)}$ and~$y_{2}^{(2)}$.

Let $\left(y_{1}^{(k)},y_{2}^{(k)},\ldots,y_{n}^{(k)}\right)^T=P_k\left(y_{1}^{(k-1)},y_{2}^{(k-1)},\ldots,y_{n}^{(k-1)}\right)^T$, where $k$-th iteration permutation matrix $P_k$ is defined as
	\[
		P_k=
		\left(
		\begin{array}{c|cccc}
		I_k & 0 & 0 & \cdots &0 \\\hline
		0&x_{k,k}^{(k)} & x_{k,k+1}^{(k)}&\cdots & x_{k,n}^{(k)}\\
		0&-x_{k+1,k}^{(k)} & 1-x_{k+1,k+1}^{(k)}&\cdots & -x_{k+1,n}^{(k)}\\
		\vdots & \vdots &\vdots &\ddots &\vdots \\
		0& -x_{n,k}^{(k)} & -x_{n,n+1}^{(k)} & \cdots & 1-x_{n,n}^{(k)}
		\end{array}
		\right).
	\]
	
This process continues until the form $g_k$ becomes diagonal. 
The center then takes the form of
	\[
	Z(f)=Z\left(c_1(y_{1}^{(k)})^d\right)\times Z\left(c_2(y_{2}^{(k)})^d\right)\times \cdots \times Z\left(c_n (y_{n}^{(k)})^d\right)
	\]
and $f$ is transformed into a sum of $n$ $d$-th powers by linear transformation
	\[
	f=c_1\left(y_{1}^{(k)}\right)^d+c_2\left(y_{2}^{(k)}\right)^d+ \cdots +c_n \left(y_{n}^{(k)}\right)^d.
	\]
Let $P=P_{n-1}P_{n-2}\cdots P_1 $ and we have
\[
    \left(y_1,y_2,\ldots,y_n\right)^T=P^{-1} \left(y_{1}^{(k)},y_{2}^{(k)},\ldots,y_{n}^{(k)}\right)^T=P^{-1}\left(z_1,z_2,\ldots,z_n\right)^T,
\]
where $y_i^{(k)}=z_i$. 
$f$ is now represented as
	\[
	f(P^{-1}z)=c_1z_1^d+c_2z_2^d+\cdots+c_nz_n^d.
	\]
    
If the original form of $f$ is degenerate, a matrix $\bar{P}$ is built as
	\[
	\bar{P}=\left (
	\begin{array}{c|c}
	P & 0\\
	\hline
	0 & I_{\bar{n}-n}
	\end{array}
	\right ).
	\]
There is a linear transformation $\left(z_1,z_2,\ldots,z_{\bar{n}}\right)^T=\bar{P}A^{-1}\left(x_1,x_2,\ldots,x_{\bar{n}}\right)^T$ between the original variable~$x_i$ and~$z_i$. Let $Q=\bar{P}A^{-1}=\left(q_{ij}\right)_{1\le i,j\le \bar{n}}$ and we have $z_i=\sum_{j=1}^{\bar{n}} q_{ij} x_j$.
Substitute this back to the sum of powers and we get
	\[
		f=c_1\left(\sum_{i=1}^{\bar{n}} q_{1i} x_i\right)^d+\cdots+c_{\bar{n}}\left(\sum_{i=1}^{\bar{n}} q_{\bar{n}i} x_i\right)^d.
	\]

The diagonal form step calculates $P$ and its inverse, which takes~$O(n^2)$ and~$O(n^3)$ operations, respectively. While the change of variables takes at most~$O(n^{2}t)$ operations, we need to expand $t$ terms of~$f$, which involves several multivariate polynomials multiplications. Since multiplying two polynomials takes $O(n\log{n}\log{\log{n}})$ operations for polynomials with $d<n$, expanding each term takes at most $O(nd\log{n}\log{\log{n}})$ operations~\cite{cantor1991fast}. Therefore, the overall complexity of this step is~$O(ndt\log{n}\log{\log{n}})$.

\begin{proposition} \label{pro:diagonalize_form}
    Diagonalizing the form of $f$ has complexity of $O(ndt\log{n}\log{\log{n}})$.
\end{proposition}

\subsection{Detect orthogonality}\label{sec:ortho_unit}

We determine the orthogonality and unitarity of the diagonalization using the permutation matrix~$P$ computed from the previous step. If $P P^{T}$ is diagonal, the diagonalization is orthogonal. If $P \bar{P}^{T}$ is diagonal, the diagonalization is unitary. 
This process involves the calculation of a transpose~$P^{T}$ and a conjugate transpose~$\bar{P}^{T}$ of matrix $P$, which both take $O(n^2)$ arithmetic operations. In addition to two matrix multiplications with complexity $O(n^3)$, the overall computational complexity of this step is~$O(n^3)$.
	

	\begin{proposition} \label{pro:orthogonal}
		Determining orthogonality and unitarity of $f$ takes $O(n^3)$ arithmetic operations.
	\end{proposition}

In summary, the iterative diagonalization process, consists of the compute center, compute idempotents, and diagonalize form steps, which have complexity of $O(n^3d^3t+n^{d+4})$, $O(n^3+\rho)$ and $O(ndt\log{n}\log{\log{n}})$, respectively. This iterative process lasts at most~$n-1$ iterations and $n$ is reduced by 1 in each iteration. Therefore, the entire process takes $O(\log n)$ costs of the three steps above. Lastly, the orthogonality and unitarity are determined straightforwardly with $O(n^3)$ operations, which does not affect the overall complexity.

\subsection{Some examples}\label{sec:d_example}

\begin{example} \label{d_e1}
	
	Consider the quadruple quintic polynomial $f$ in Example~\ref{d_c1}, whose diagonalizability has been verified. Build and solve system $N$ using its center $Z(f)$ with rank 1 and trace 1 conditions, which may give more than one solution. Choose one solution as $\epsilon_1$ with nonzero 1,1-th entry and obtain a pair of orthogonal idempotents $(\epsilon_1,I_4-\epsilon_1)$, where
	\begin{equation*}
		\epsilon_1 = \begin{pmatrix} \frac{1}{2} & \frac{1}{2} & -\frac{1}{2} & -1 \\ \frac{3}{2} & \frac{3}{2} & -\frac{3}{2} & -3 \\ 0 & 0 & 0 & 0 \\ \frac{1}{2} & \frac{1}{2} & -\frac{1}{2} & -1 \end{pmatrix}, \quad 
        I_4-\epsilon_1 = \begin{pmatrix} \frac{1}{2} & -\frac{1}{2} & \frac{1}{2} & 1 \\ -\frac{3}{2} & -\frac{1}{2} & \frac{3}{2} & 3 \\ 0 & 0 & 1 & 0 \\ -\frac{1}{2} & -\frac{1}{2} & \frac{1}{2} & -2 \end{pmatrix}.
	\end{equation*}
	According to the pair $(\epsilon_1,I_4-\epsilon_1)$, we take a change of variables $y=P_{1}x$, where
	\begin{equation*}
		P_{1} = \begin{pmatrix} \frac{1}{2} & \frac{1}{2} & -\frac{1}{2} & -1 \\ -\frac{3}{2} & -\frac{1}{2} & \frac{3}{2} & 3 \\ 0 & 0 & 1 & 0 \\ -\frac{1}{2} & -\frac{1}{2} & \frac{1}{2} & 2 \end{pmatrix}.
	\end{equation*}
	By direct computation we have
	\begin{align*}
		f(P^{-1}_{1}y)&=1024y_1^5+7807y_2^5-51915y_2^{4}y_3-64955y_2^{4}y_4+138310y_2^{3}y_3^2+345900y_2^3y_3y_4+216310y_2^3y_4^2\\
        &-184350y_2^2y_3^3-691410y_2^2y_3^2y_4-864450y_2^2y_3y_4^2-360310y_2^2y_4^3+122885y_2y_3^4+614460y_2y_3^3y_4\\
        &+1152210y_2y_3^2y_4^2+960300y_2y_3y_4^3+300155y_2y_4^4-32768y_3^5-204805y_3^4y_4-512030y_3^3y_4^2\\
        &-640070y_3^2y_4^3-400075y_3y_4^4-100031y_4^5.
	\end{align*}
Note that the decomposition takes the form of $f(P^{-1}_{1}y) = t(y_1)+g(y_2,y_3,y_4)$, where~$t(y_1)=-1024y_1^5$. Since~$g$ is not diagonal, we continue to diagonalize $g$. Calculate $Z(g)=\{X_g \in \C^{3\times 3}|(H_{g}X_g)^T=X_gH_{g}\}$ and the general solution is
	\begin{equation*}
	X_g = \lambda_{1}\begin{pmatrix} -8 & 8& 13 \\ -16 & 11 & 16 \\ 1 & 0 & 0 \end{pmatrix}+\lambda_{2}\begin{pmatrix} -16 & 11 & 16 \\ -12 & 2 & 12 \\ 0 & 1 & 0 \end{pmatrix}+\lambda_{3}\begin{pmatrix} 1 & 0 & 0\\ 0 & 1 & 0 \\ 0 & 0 & 1 \end{pmatrix}.
	\end{equation*}
	We obtain a new pair orthogonal idempotents $(\epsilon_2, I_3-\epsilon_2)$ based on $X_g$, where
	\begin{equation*}
		\epsilon_2=\begin{pmatrix} \frac{3}{2} & -\frac{3}{2} & -\frac{3}{2} \\ -2 & 2 & 2 \\ \frac{5}{2} & -\frac{5}{2} & -\frac{5}{2} \end{pmatrix}, \quad I_3-\epsilon_2=\begin{pmatrix} -\frac{1}{2} & \frac{3}{2} & \frac{3}{2} \\ 2 & -1 & -2 \\ -\frac{5}{2} & \frac{5}{2} & \frac{7}{2} \end{pmatrix}.
	\end{equation*}
We then take a change of variables $z=P_{2}y$ according to $(\epsilon_2,I_3-\epsilon_2)$, where 
	\begin{equation*}
	P_{2} = \begin{pmatrix} 1 & 0 & 0 & 0 \\0 & \frac{3}{2} & -\frac{3}{2} & -\frac{3}{2} \\ 0 & 2 & -1 & -2 \\ 0 & -\frac{5}{2} & \frac{5}{2} & \frac{7}{2} \end{pmatrix}.
	\end{equation*}
	
	By direct computation we have
	\begin{equation*}
		g(P_{2}^{-1}z)=-0.13162z_2^5-31z_3^5-320z_3^4z_4-1280z_3^3z_4^2-2560z_3^2z_4^3-
	2560z_3z_4^4-1024z_4^5.
	\end{equation*}
	Note that the decomposition takes the form of $g(P^{-1}_{2}z) = t(z_2)+h(z_3,z_4)$, where~$t(z_2)=-0.13162z_2^5$. Since~$h$ is not diagonal, we continue to diagonalize $h$. Compute $Z(h)=\{X_h \in \C^{2\times 2}|(H_{h}X_h)^T=X_hH_{h}\}$ and the general solution is
	\begin{equation*}
		X_h = \lambda_{1}\begin{pmatrix} -2 & 0\\ 1 & 0 \end{pmatrix}+\lambda_{2}\begin{pmatrix} 1 & 0\\ 0 & 1 \end{pmatrix}.
	\end{equation*}
Compute orthogonal idempotents $(\epsilon_3, I_2-\epsilon_3)$, we have
\begin{equation*}
    \epsilon_3=\begin{pmatrix} 1 & 0\\ -\frac{1}{2} & 0 \end{pmatrix}, \quad I_2-\epsilon_3=\begin{pmatrix} 0 & 0\\ \frac{1}{2} & 1 \end{pmatrix}.
    \end{equation*}
Take a change of variables $w=P_{3}z$ according to the pair ($\epsilon_3,I_2-\epsilon_3$), 
	\begin{equation*}
		P_{3} = \begin{pmatrix} 1 & 0 & 0 & 0 \\ 0 & 1 & 0 & 0 \\ 0 & 0 & 1 & 0 \\ 0 & 0 & \frac{1}{2} & 1 \end{pmatrix}.
	\end{equation*}
	By direct computation, we have the following decomposition
	\[
	h(P^{-1}_{3}w)=w_3^5-1024w_4^5.
	\]
Since $h$ is in diagonal form, the diagonalization process terminates. The final change of variables~$w=Px$ takes the form of
\[
	P = P_3P_2P_1 = \begin{pmatrix} \frac{1}{2} & \frac{1}{2} & -\frac{1}{2} & -1 \\ -\frac{3}{2} & 0 & 0 & \frac{3}{2} \\-2 & 0 & 1 & 2 \\ 1 & -\frac{1}{2} & 1 & \frac{1}{2} \end{pmatrix}
\]
and we obtain
\[
    f(P^{-1}w)=1024w_1^5-0.13162w_2^5+w_3^5-1024w_4^5.
\]
Finally, $f$ is diagonalized as
    \[
        f\left(x_1,x_2,x_3,x_4\right) = \left(2x_1+2x_2-2x_3-4x_4\right)^5+\left(x_1-x_4\right)^5+\left(-2x_1+x_3+2x_4\right)^5+\left(-4x_1+2x_2-4x_3+2x_4\right)^5.
    \]
In addition, after examining the permutation matrix $P$, we deduce that the diagonalization of $f$ is neither orthogonal nor unitary.

\end{example}
	
\begin{example}\label{d_e2}
Consider the sextuple septic polynomial $f$ defined in Example~\ref{d_c2}, whose diagonalizability has been verified. Build and solve system $N$ based on $Z(f)$, we calculate a pair of orthogonal idempotents $(\epsilon_1,I_6-\epsilon_1)$, where
	\begin{equation*}
		\epsilon_1 = \begin{pmatrix} \frac{4}{9} & 0 & \frac{4}{9} & 0 & 0 & 0 \\ -\frac{8}{9} & 0 & -\frac{8}{9} & 0 & 0 & 0 \\ \frac{5}{9} & 0 & \frac{5}{9} & 0 & 0 & 0 \\ -\frac{4}{9} & 0 & -\frac{4}{9} & 0 & 0 & 0 \\ -\frac{2}{9} & 0 & -\frac{2}{9} & 0 & 0 & 0 \\ \frac{1}{3} & 0 & \frac{1}{3} & 0 & 0 & 0 \\ \end{pmatrix}, \quad
        I_6-\epsilon_1 = \begin{pmatrix} \frac{5}{9} & 0 & -\frac{4}{9} & 0 & 0 & 0 \\ \frac{8}{9} & 1 & \frac{8}{9} & 0 & 0 & 0 \\ -\frac{5}{9} & 0 & \frac{4}{9} & 0 & 0 & 0 \\ \frac{4}{9} & 0 & \frac{4}{9} & 1 & 0 & 0 \\ \frac{2}{9} & 0 & \frac{2}{9} & 0 & 1 & 0 \\ -\frac{1}{3} & 0 & -\frac{1}{3} & 0 & 0 & 1 \\ \end{pmatrix}.
	\end{equation*}
	According to $(\epsilon_1,I_6-\epsilon_1)$, we take a change of variables $y=P_{1}x$, where
	\begin{equation*}
		P_{1} = \begin{pmatrix} \frac{4}{9} & 0 & \frac{4}{9} & 0 & 0 & 0 \\ \frac{8}{9} & 1 & \frac{8}{9} & 0 & 0 & 0 \\ -\frac{5}{9} & 0 & \frac{4}{9} & 0 & 0 & 0 \\ \frac{4}{9} & 0 & \frac{4}{9} & 1 & 0 & 0 \\ \frac{2}{9} & 0 & \frac{2}{9} & 0 & 1 & 0 \\ -\frac{1}{3} & 0 & -\frac{1}{3} & 0 & 0 & 1 \end{pmatrix}
	\end{equation*}
	and
\begin{align*}
	f(P^{-1}_{1}y)&=-291.9293y_1^7+1y_2^7-14y_2^6y_3+84y_2^5y_3^2-280y_2^4y_3^3+560y_2^3y_3^4-672y_2^2y_3^5+448y_2y_3^6-128y_3^7\\
    &-7y_3^6y_4+14y_3^6y_5-21y_3^6y_6+21y_3^5y_4^2+126y_3^5y_4y_6-84y_3^5y_5^2+189y_3^5y_6^2-35y_3^4y_4^3-315y_3^4y_4^2y_6\\
    &-945y_3^4y_4y_6^2+280y_3^4y_5^3-945y_3^4y_6^3+35y_3^3y_4^4+420y_3^3y_4^3y_6+1890y_3^3y_4^2y_6^2+3780y_3^3y_4y_6^3\\
    &-560y_3^3y_5^4+2835y_3^3y_6^4-21y_3^2y_4^5-315y_3^2y_4^4y_6-1890y_3^2y_4^3y_6^2-5670y_3^2y_4^2y_6^3-8505y_3^2y_4y_6^4+672y_3^2y_5^5\\
    &-5103y_3^2y_6^5+7y_3y_4^6+126y_3y_4^5y_6+945y_3y_4^4y_6^2+3780y_3y_4^3y_6^3+8505y_3y_4^2y_6^4+10206y_3y_4y_6^5\\
    &-448y_3y_5^6+5103y_3y_6^6-2y_4^7+14y_4^6y_5-21y_4^6y_6-84y_4^5y_5^2-189y_4^5y_6^2+280y_4^4y_5^3-945y_4^4y_6^3\\
    &-560y_4^3y_5^4-2835y_4^3y_6^4+672y_4^2y_5^5-5103y_4^2y_6^5-448y_4y_5^6-5103y_4y_6^6+2443y_5^7+10206y_5^6y_6\\
    &+20412y_5^5y_6^2+22680y_5^4y_6^3+15120y_5^3y_6^4+6048y_5^2y_6^5+1344y_5y_6^6-2059y_6^7.
\end{align*}
Note that the decomposition takes the form of $f(P^{-1}_{1}y) = t(y_1)+g(y_2,\ldots,y_6)$, where~$t(y_1)=-291.9293y_1^7$. Since~$g$ is not diagonal, we continue to diagonalize $g$. Repeat the above process for another 4 iterations and we obtain the following $P_i$ for changes of variables $y^{(i+1)}=P_{i}y^{(i)}$ for $i=2,\ldots,5$, where 
	\[
	P_{2} = \begin{pmatrix} 1 & 0 & 0 & 0 & 0 & 0 \\ 0 & 1 & -2 & 0 & 0 & 0 \\ 0 & 0 & 1 & 0 & 0 & 0 \\ 0 & 0 & 0 & 1 & 0 & 0 \\ 0 & 0 & 0 & 0 & 1 & 0 \\ 0 & 0 & 0 & 0 & 0 & 1 \end{pmatrix}, \quad P_{3} = \begin{pmatrix} 1 & 0 & 0 & 0 & 0 & 0 \\ 0 & 1 & 0 & 0 & 0 & 0 \\ 0 & 0 & \frac{4}{9} & -\frac{4}{9} & 0 & -\frac{4}{3} \\ 0 & 0 & -\frac{4}{9} & \frac{13}{9} & 0 & \frac{4}{3} \\ 0 & 0 & -\frac{2}{9} & \frac{2}{9} & 0 & \frac{2}{3} \\ 0 & 0 & \frac{1}{3} & -\frac{1}{3} & 0 & 0 \end{pmatrix},
	\]
	\[
	\quad P_{4} = \begin{pmatrix} 1 & 0 & 0 & 0 & 0 & 0 \\ 0 & 1 & 0 & 0 & 0 & 0 \\ 0 & 0 & 1 & 0 & 0 & 0 \\ 0 & 0 & 0 & -\frac{4}{9} & \frac{8}{9} & \frac{4}{3} \\ 0 & 0 & 0 & \frac{2}{9} & \frac{5}{9} & \frac{2}{3} \\ 0 & 0 & 0 & -\frac{1}{3} & \frac{2}{3} & 0 \end{pmatrix}, \quad P_{5} = \begin{pmatrix} 1 & 0 & 0 & 0 & 0 & 0 \\ 0 & 1 & 0 & 0 & 0 & 0 \\ 0 & 0 & 1 & 0 & 0 & 0 \\ 0 & 0 & 0 & 1 & 0 & 0 \\ 0 & 0 & 0 & 0 & 1 & \frac{2}{3} \\ 0 & 0 & 0 & 0 & 0 & 1 \end{pmatrix}.
	\]
 
The final change of variables~$z=y^{(5)}=Px$ then takes the form of
\[
    P = P_5P_4P_3P_2P_1 = \begin{pmatrix} \frac{4}{9} & 0 & \frac{4}{9} & 0 & 0 & 0 \\ 2 & 1 & 0 & 0 & 0 & 0 \\ 0 & 0 & \frac{4}{9} & -\frac{4}{9} & 0 & -\frac{4}{3} \\ \frac{4}{9} & 0 & 0 & 0 & \frac{8}{9} & 0 \\ 0 & 0 & 0 & 0 & 1 & \frac{2}{3} \\ 0 & 0 & 0 & -\frac{1}{3} & \frac{2}{3} & 0 \\ \end{pmatrix}
\]
and we obtain
\[
    f(P^{-1}z)= -291.929z_1^7+z_2^7+291.929z_3^7+291.929z_4^7+2187z_5^7+2187z_6^7.
\]
Lastly, $f$ is diagonalized as
\[
    f\left(x_1,\ldots,x_6\right) = \left(-x_1-x_3\right)^7+\left(2x_1+x_2\right)^7+\left(x_3-x_4-3x_6\right)^7+\left(x_1+2x_5\right)^7+\left(3x_5+2x_6\right)^7+\left(-x_4+2x_5\right)^7.
\]
	The examination of $P$ shows that this diagonalization is neither orthogonal nor unitary.

\end{example}

\section{Comparison to other diagonalization techniques}\label{sec:compare}

Current diagonalization techniques, including~\cite{fedorchuk2020direct, kayal2011efficient, koiran2023absolute, koiran2021derandomization, robeva2016orthogonal} have successfully provided criteria and algorithms for the diagonalizations of homogeneous polynomials.
 While they have made huge progress in the field, they are either only applicable to specific types of polynomials or appeal to highly nonlinear tasks such as polynomial factorization. We provide more discussions of their complexity below.
	
Kayal~\cite{kayal2011efficient} did some of the early work in this field and proposed a criterion and diagonalization algorithm for $f$ with arbitrary $n$ and $d$.
It is based on the factorization properties of the Hessian determinant~$\text{DET}(H)$ of $f$, which admits randomized polynomial complexity. The procedures are summarized in Algorithm~\ref{alg:kayal}. At step 1, determinant~$\text{DET}(H)$ may be factorized using Kaltofen's algorithm~\cite{kaltofen1989factorization} or the black box factorization algorithm of Kaltofen and Trager~\cite{kaltofen1990computing}, which both admits randomized polynomial time. 
However, Kaltofen's algorithm uses Hensel lifting for $d$ iterations, which leads to an exponential blowup of sizes in modules where the previous computations cannot be reused~\cite{sinhababu2021factorization}. While more work has been done to improve its efficiency for high-degree polynomials with a sparse representation, the complexity is still quite high for large $d$. At step 2, the coefficients~$a_i$ can be obtained efficiently by solving a dense linear algebra problem. Similarly, the complexity of root findings at step 3 is polynomial and the equivalence problem is solved in randomized polynomial time.
    
\begin{algorithm}
		\caption{Kayal's algorithm}\label{alg:kayal}
		\begin{algorithmic} [1]
		\State Compute Hessian determinant~$\text{DET}(H)$ of~$f$ and check if it is identically 0 and can be factorized as $\text{DET}(H)=c\prod_{i=1}^{n}l_i(x_1,x_2, \ldots,x_n)^{d-2}$ where~$l_i$ are linear forms and~$c\in \C$. Reject if it is not possible.
		\State Calculate constants~$a_i \in \C$ such that~$f(x_1,x_2,\ldots,x_n)=\sum_{i=1}^{n}\lambda_{i}l_{i}(x_1,x_2,\ldots,x_n)^{d}$. Terminate if it is not possible.
		\State Check if all~$a_i$ values have $d$-th roots in $\C$. Reject if it is not possible.
		\end{algorithmic}
\end{algorithm}


Koiran~\cite{koiran2021derandomization} proposed a deterministic criterion to determine the diagonalizability for polynomimals~$f$ with degree~$d=3$ as described in Algorithm~\ref{alg:koiran}. In contrast to Kayal's technique, it does not factorize Hessian determinant~$\text{DET}(H)$ explicitly. Recall that a homogeneous polynomial~$f$ with $d=3$ can be associated to a symmetric tensor~$T$ of order 3, where~$f(x_1,x_2,\ldots,x_n)=\sum_{i,j,k=1}^{n}T_{ijk}x_ix_jx_k$. We have~$\frac{\partial^3f}{\partial x_i \partial x_j \partial x_k} = 6T_{ijk}$ and the~$i$-th slice of~$T$ is the symmetric matrix~$T_{i}$ with entries~$(T_i)_{jk}=T_{ijk}$.

\begin{algorithm}
		\caption{Koiran's algorithm}\label{alg:koiran}
		\begin{algorithmic} [1]
		\State On input $f\in \C[x_1,x_2,\ldots,x_n]$, pick a random matrix $R\in M_n(\C)$ and set $h(x) = f (Rx)$.
		\State Let $T_1,T_2,\ldots,T_n$ of $h$ be the slices of $h$. If $T_1$ is singular, reject. Otherwise, compute $T_{1}'=T_1^{-1}$.
		\State If the matrices $T_{1}'T_k$ commute and are all diagonalizable over $\C$, accept. Otherwise, reject.
		\end{algorithmic} 
\end{algorithm}
	
Koiran's criterion iterates through~$n$ slices~$T_i$ to determine their diagonalizability. Instead of factorizing explicitly, Koiran used the existence of the factorization to find deterministically a point where $H$ does not vanish. It includes routines such as calculations of roots from polynomials, which can be obtained through Hurwitz determinants. Koiran and Saha~\cite{koiran2023absolute} provided further improvement and complexity analysis on this algorithm, which showed that it requires $O(n^{\omega+2})$ arithmetic operations over the complex field and $\omega$ is a feasible exponent for matrix multiplication. Koiran and Saha also presented a new randomized black-box criterion that uses the extraction of complex polynomial roots to determine the diagonalizability of~$f$ with~$d>3$. This algorithm takes $O(n^{2}d\log^2{d}\log{\log{d}}+n^{\omega+1})$ arithmetic operations for $f\in\C$.
 
Robeva~\cite{robeva2016orthogonal} studied the diagonalization from the perspective of symmetric tensors and proposed to use the tensor power method from Anandkumar et al.~\cite{anandkumar2014tensor} or the tensor decomposition technique from Brachat et al.~\cite{brachat2010symmetric}, which is based on the properties of Hankel matrices. However, its algorithm is difficult to construct and requires solving a large nonlinear system of equations.

Fedorchuk~\cite{fedorchuk2020direct} interpreted the diagonalizability in terms of factorization properties of the Macaulay inverse system of its Milnor algebra as shown in Algorithm~\ref{alg:fedorchuk}. It firstly determines whether given $f$ is smooth and then obtains an associated form~$A(f)$ of $f$ by computing the degree $n(d-1)$ part of the Gr\"{o}bner basis of the Jacobian ideal $J_f$. This reduces the diagonalization to a polynomial factorization problem, which has polynomial time complexity. Fedorchuk also stated that step 2 of Algorithm~\ref{alg:fedorchuk} is computationally expensive for large~$n$ and~$d$.

\begin{algorithm}
		\caption{Fedorchuk's algorithm}\label{alg:fedorchuk}
		\begin{algorithmic} [1]
		\State Compute $J_f = \left(\frac{\partial f}{\partial x_1},\frac{\partial f}{\partial x_2},\ldots,\frac{\partial f}{\partial x_n}\right)$ up to degree $n(d-1)+1$. If $(J_f)_{n(d-1)+1}\neq \C[x_1,x_2,\ldots,x_n]_{n(d-1)+1}$, then $f$ is not smooth and we stop; otherwise, continue.
		\State Compute $A(f)$ as the dual to $(J_f)_{n(d-1)+1}$ and the irreducible factorization of $A(f)$ in $\C[x_1,x_2,\ldots,x_n]_{n(d-1)+1}$ and check for the existence of balanced direct product factorizations. If any exist, then $f$ is a direct sum; otherwise, $f$ is not a direct sum.
		\State Decompose $f$ using basis $V$ from computed factorization of $A(f)$.
		\end{algorithmic}
\end{algorithm}



Some other studies focused on determining properties of diagonalization, such as orthogonality and unitarity~\cite{koiran2021orthogonal,kolda2000orthogonal,boralevi2017orthogonal,robeva2016orthogonal}. These properties are useful for areas like theoretical computer science and scientific computing. Kolda~\cite{kolda2000orthogonal} studied the orthogonal decomposition of tensors and computed the best rank-1 approximations by solving a minimization problem with orthogonality constraints. Kolda used an alternating least squares approach to solve the problem, which converges slowly and becomes a computationally challenging task. Robeva~\cite{robeva2016orthogonal} tackled the issue of orthogonal decompositions of symmetric tensors using the tensor power method. Boralevi et al.~\cite{boralevi2017orthogonal} proposed a randomized algorithm to check the orthogonality and unitarity of a tensor and decompose it if possible, which is based on singular value decompositions. Koiran~\cite{koiran2021orthogonal} expanded the study for orthogonal and unitary decompositions and over the field of complex numbers. Their techniques need to solve a large system of quadratic equations. In contrast, Huang et al.'s diagonalization algorithm only requires basic matrix inversions and multiplications to determine the orthogonality and unitarity of the diagonalization using computed permutation matrix $P$.

To conclude, many techniques have been proposed as criteria for diagonalizability, diagonalization, or criteria for orthogonality and unitarity. They either only apply to limited polynomials or use computationally expensive techniques. While some researchers provided complexity analysis, few gave sufficient numerical results for illustration and comparison. In contrast, our complexity analysis is validated by numerical experiments with concrete examples. This makes the center-based diagonalization algorithm a competitive alternative in this field.

\section{Summary}\label{sec:summary}

In this article, we provide detailed formulations and complexity analysis of the diagonalization technique based on Harrison's center theory. It consists of a criterion for diagonalizability and a diagonalization algorithm. The technique is divided into six steps and their computational complexity is summarized in Table~\ref{tab:complexity}. This table also contains the convergence rates from the numerical results given in Appendix~\ref{app:cost}. The convergence is measured for $f$ with various~$n$ and~$d$ values, which reflect given problem sizes. 


Among the six steps, the detect nondegeneracy, compute center, and diagonalize form steps have the highest theoretical complexity. Their costs depend on the number of terms~$t$ in~$f$, which increases if either $n$ or $d$ increases. The detect nondegeneracy and compute center steps need to compute multiple first- or second-order derivatives of $f$. The diagonalize form step applies the change of variables to~$f$ and requires operations like polynomial multiplications or power expansions. Their complexity is consistent with their high convergence rates for~$n$ and $d$ as shown in numerical results.

\begin{center}
	\begin{table}%
	\centering
	\caption{Complexity of each step of diagonalization\label{tab:complexity}}%
		\begin{tabular*}{450pt}{@{\extracolsep\fill}lccc@{\extracolsep\fill}}%
		\toprule
		\textbf{Step} & \textbf{Complexity} & \textbf{Convergence $n$} & \textbf{Convergence $d$} \\
		\midrule
		detect nondegeneracy & $O(n^2t)$ & $O(n^{4.02})$ & $O(d^{2.50})$ \\
		compute center & $O(n^3d^3t+n^{d+4})$ & $O(n^{6.07})$ & $O(d^{2.85})$ \\
		detect diagonalizability & $O(n^5)$ & $O(n^{1.31})$ & $O(d^{-0.020})$ \\
		compute idempotent & $O(n^3+\rho)$ & $O(n^{2.5})$ & $O(d^{0.004})$ \\
		diagonalize form & $O(ndt\log{n}\log{\log{n}})$ & $O(n^{5.12})$ & $O(d^{2.93})$ \\
		detect orthogonality & $O(n^3)$ & $O(n^{0.062})$ & $O(d^{-0.40})$ \\
  		total & & $O(n^{4.26})$ & $O(d^{2.39})$ \\
		\bottomrule
		\end{tabular*}
	\end{table}
\end{center}
    
The complexity of the other three steps is independent of~$t$ and shows markedly lower convergence rates for $n$ and $d$. While we solve a quadratic system of equations~$N$ to compute idempotents, the system's dimension only depends on~$n$ and its entries are relatively sparse. Consequently, the convergence of its computational costs for $n$ is below quadratic and remains relatively unchanged when $d$ increases. Similarly, the increase in $d$ does not affect the costs of the detect diagonalizability step.

The criterion of diagonalizability consists of detect nondegeneracy, compute center and detect diagonalizability steps and the complexity of this process is 
    \[
        O(n^2t)+O(n^3d^3t+n^{d+4})+O(n^5) = O(n^3d^3t+n^{d+4}).
    \]
Since $t$ grows polynomially if either~$n$ or $d$ increases while the other is fixed. We deduce the complexity of this criterion is polynomially bounded for $n$ and $d$. The diagonalization process comprises an iterative process, including the compute center, compute idempotent, and diagonalize form steps, and terminates with at most $n-1$ steps. Consequently, its complexity is
	\[
		O(\log n)\left(O(n^3d^3t+n^{d+4})+ O(n^3+\rho) + O(ndt\log{n}\log{\log{n}})\right).
	\]
Since we cannot give an accurate bound for solving quadratic system $N$, its complexity may not be polynomially bounded. However, its computation is performed efficiently in the numerical experiments and it is shown that $O(\rho)$ is below quadratic for $n \le 7$.

We compared our technique with other criteria and diagonalization techniques regarding their computational costs. Many of them lack detailed complexity analysis and rely on methods such as the polynomial factorization and tensor power method, which require randomized polynomial time. Some other techniques like~\cite{kayal2011efficient,koiran2021derandomization,koiran2023absolute} may fail with a small probability. In contrast, the center-based diagonalization technique only requires solving sparse linear and quadratic systems and is applicable for $f$ with arbitrary $n$ and $d$.

\subsection{Future work}\label{sec:future}
The exact complexity of solving its quadratic systems of equations is still unknown and needs to be addressed in future studies. Additionally, the numerical experiments showed that the rounding errors accumulated during the iterative process may affect the accuracy of the solution. Future studies should further investigate the algorithm's stability and error bounds, which are crucial to the practicability of this technique. Besides, current research mostly focuses on one polynomial. However, real-world applications may demand a simultaneous criterion and diagonalization of more than one polynomial. We also want to extend the complexity analysis to the direct sum decomposition and Waring decomposition of polynomials. 




\th{Conflict of Interest} The authors declare that they have no known competing financial interests or personal relationships that could have appeared to influence the work reported in this paper.

\appendix

\section{Computational cost of diagonalization}\label{app:cost}
In this section, we provide numerical results to validate the impacts of $n$ and $d$ on the costs of the diagonalization technique. All the numerical experiments are conducted in MATLAB R2022a on a personal computer with Microsoft Windows 10 operating system and an Intel$^\circledR$ Core\textsuperscript{TM} i9-10900 CPU @ 2.80 GHz configuration with 64.0 GB of RAM. The implementations are developed by Fang and available from Bitbucket\footnote{https://bitbucket.org/fanglishan/diagonalization/src/master/}. The computational costs for various dimensions~$n$ and degrees~$d$ are given in Examples~\ref{app_e1} and~\ref{app_e2}, respectively.

\begin{example} \label{app_e1}
Consider a homogeneous polynomial
\[
    f = \sum_{i=1}^{n}\left(\lambda_{i1}x_1 + \cdots + \lambda_{in}x_n\right)^5,
\]
where~$\lambda_{ij}$ are randomly generated such that $\{\lambda_{i1},\lambda_{i2},\ldots,\lambda_{in}\}$ are linearly independent. A numerical experiment was conducted to compare the computational costs for diagonalizing~$f$ with~$n=2,\ldots,7$, and the statistics are shown in Table~\ref{tab:time_n}. It includes the execution time of the six major steps of Algorithm~\ref{alg:diagonalization} and the entire program for different~$n$ values. The execution time was estimated as the average elapsed time of ten executions of the diagonalization program, measured in seconds. Table~\ref{tab:time_n} also includes their convergence rates to verify the previous complexity analysis. Note that compute center, compute idempotent, and diagonalize form steps may be executed several times during the iterative process and Table~\ref{tab:time_n} only shows the execution time in the first iteration.
 
Table~\ref{tab:time_n} shows the total execution time of the diagonalization increases in the order of $O(n^{4.26})$, which is markedly lower than the theoretical upper bounds given in Section~\ref{sec:summary}. Among the six steps, the compute center step has the highest convergence rate of $O(n^{6.07})$. While the diagonalize form step has a lower convergence of $O(n^{5.12})$, its execution time in the first iteration takes up over $38\%$ of the total execution time for $n=7$. Since both of these steps involve iterating through $t$ terms of $f$, its costs increase significantly if either $n$ or $d$ increases. 

While the detect nondegeneracy step has a high convergence rate of $O(n^{4.02})$, it only takes 0.14 seconds for $n=7$ and is trivial to the algorithm. The other steps like the detect diagonalizability, compute idempotent, and detect orthogonality all have low convergence compared to their upper bounds. Since the detect orthogonality step only involves a single matrix inversion and multiplication, it has a low impact on the program execution time for all $n$ and shows near-zero convergence.

The computational costs for solving linear system $CX_v =0$ and quadratic system $N$ are important components of the complexity analysis as they are not easily parallelizable. We provide additional execution time and corresponding convergence rates for matrix~$C$ and system~$N$ in Table~\ref{tab:stat_n}, including the solve time, dimension, and sparsity. The dimension refers to the number of equations in the system and the sparsity is the ratio of nonzero entries. Table~\ref{tab:stat_n} shows that the costs for computing the RREF for matrix~$C$ increase drastically in the order of $O(n^{6.66})$. This is reasonable as $C$'s dimension increases with convergence $O(n^{4.50})$ and sparsity only decreases in the order of $O(n^{-1.08})$. In contrast, the costs for solving system~$N$ increase in the order of $O(n^{1.85})$. While the dimension of system~$N$ increases quadratically as shown in Section~\ref{sec:idempotent}, its quadratic equations only involve four variables and were efficiently solved in the experiments.

\end{example}

    \begin{sidewaystable}
    \centering
    \caption{Execution time of diagonalization with various $n$}\label{tab:time_n}
    \begin{tabular*}{\textheight}{@{\extracolsep\fill}lccccccc}
		\toprule
		\textbf{Step\,\textbackslash \,\,$n$} & 2 & 3 & 4 & 5 & 6 & 7 & \textbf{Convergence} \\
		\midrule
		detect nondegeneracy & $9.00\times10^{-4}$ & $1.40\times10^{-3}$ & $4.90\times10^{-3}$ & 0.015 & 0.037 & 0.14 & 4.02 \\
		compute center & $1.10\times10^{-3}$ & $4.50\times10^{-3}$ & 0.031 & 0.13 & 0.59 & 1.92 & 6.07 \\
		detect diagonalizability & 0.022 & 0.036 & 0.046 & 0.056 & 0.082 & 0.13 & 1.31 \\
		detect idempotent & 0.099 & 0.19 & 0.39 & 0.72 & 1.22 & 2.31 & 2.50 \\
		diagonalize form & 0.047 & 0.71 & 1.64 & 4.91 & 16.21 & 36.86 & 5.12 \\
		check orthogonality & $1.00\times10^{-3}$ & $1.00\times10^{-3}$ & $1.00\times10^{-3}$ & $1.00\times10^{-3}$ & $1.20\times10^{-3}$ & $1.00\times10^{-3}$ & 0.062 \\
		total & 0.41 & 1.95 & 5.17 & 13.32 & 37.91 & 96.22 & 4.26 \\
		\bottomrule
    \end{tabular*}
    
    \vspace*{20pt}
    
    \caption{Statistics of matrix $C$ and $N$ with various $n$}\label{tab:stat_n}
    \begin{tabular*}{\textheight}{@{\extracolsep\fill}lccccccc}
		\toprule
		\textbf{Step\,\textbackslash \,\,$n$} & 2 & 3 & 4 & 5 & 6 & 7 & \textbf{Convergence} \\
		\midrule
		$C$ solve time & $1.00\times10^{-4}$ & $1.30\times10^{-3}$ & 0.010 & 0.039 & 0.13 & 0.46 & 6.66 \\
		$C$ dimension & 13 & 71 & 261 & 736 & 1667 & 3613 & 4.50 \\
		$C$ sparsity & 1.00 & 0.67 & 0.47 & 0.40 & 0.27 & 0.28 & -1.08 \\
		$N$ solve time & 0.055 & 0.068 & 0.10 & 0.15 & 0.33 & 0.58 & 1.85 \\
		$N$ dimension & 4 & 11 & 22 & 37 & 56 & 79 & 2.38 \\
		\bottomrule
    \end{tabular*}
        
    \vspace*{20pt}
    
    \caption{Execution time of diagonalization with various $d$}\label{tab:time_d}
    \begin{tabular*}{\textheight}{@{\extracolsep\fill}lcccccccc}
		\toprule
		\textbf{Step\,\textbackslash \,\,$d$} & 5 & 10 & 20 & 30 & 40 & 50 & \textbf{Convergence} \\
		\midrule
		detect nondegeneracy & $2.10\times10^{-3}$ & $7.30\times10^{-3}$ & 0.042 & 0.14 & 0.29 & 0.56 & 2.50 \\
		compute center & $5.70\times10^{-3}$ & 0.027 & 0.18 & 0.66 & 1.76 & 3.98 & 2.85 \\
		detect diagonalizability & 0.034 & 0.033 & 0.032 & 0.032 & 0.031 & 0.033 & -0.020 \\
		compute idempotent & 0.18 & 0.18 & 0.18 & 0.19 & 0.18 & 0.18 & $4.10\times10^{-3}$ \\
		diagonalize form & 0.45 & 0.97 & 4.28 & 29.62 & 109.25 & 398.64 & 2.93 \\
		detect orthogonality & $1.00\times10^{-3}$ & $1.00\times10^{-3}$ & $1.10\times10^{-3}$ & $1.20\times10^{-3}$ & $1.00\times10^{-3}$ & $1.00\times10^{-3}$ & -0.40 \\
		Total & 1.58 & 3.85 & 13.69 & 50.16 & 145.09 & 454.08 & 2.39 \\
		\bottomrule
    \end{tabular*}
    
    \end{sidewaystable}

	\begin{example} \label{app_e2}
	Consider a homogeneous polynomial
	\[
		f = \left(x_1 + 2x_2 - x_3\right)^d + \left(2x_1 +2x_2-x_3\right)^d + \left(x_1 - 2x_2 + 2x_3\right)^d,
	\]
where~$d \in \Z^{+}$. We tested the computational costs of the diagonalization algorithm for various degrees $d$. We display statistics for $d=5, 10, 20, 30, 40$ and $50$ in Table~\ref{tab:time_d}, which includes the execution time and convergence rates of each step and the entire program. The execution time of the program increases in the order of $O(d^{2.39})$, which is markedly lower compared to that for $n$ in Example~\ref{app_e1}. Among the six steps, the costs of the compute center and diagonalize form steps increase near cubic order and are significantly higher than the other steps. The theoretical analysis shows that the complexity of the algorithm depends on~$t$, which increases as~$d$ increases as demonstrated in the experiment.
 
In contrast, the costs of the other four steps are similar for different $d$. The upper bounds of their complexity are $O(n^5)$, $O(n^3+\rho)$, and $O(n^3)$, respectively. Since~$d$ does not affect the dimension and solve time for matrices~$B$ and~$N$, it has a trivial influence on their execution time. This example also shows the center-based diagonalization algorithm works for arbitrary $d$.

\end{example}
    
The numerical results in Examples~\ref{app_e1} and~\ref{app_e2} demonstrate that the diagonalization algorithm runs in polynomial time if either~$n$ or~$d$ is fixed. The convergence of each step of the algorithm is generally consistent with the complexity analysis based on factors~$n$,~$d$ and~$t$. In addition, the time-consuming steps like the diagonalize form step are highly parallelizable and the parallel complexity of their polynomial multiplication can be reduced to $O(\log{n})$~\cite{cantor1991fast}.
    

\section{Number of polynomial terms}\label{app:no_terms}

The number of terms~$t$ in polynomial $f$ is a critical factor in the complexity analysis of the diagonalization algorithm. We show that the criterion of diagonalizabiilty has polynomial time complexity based on $n$, $d$, and $t$ with additional numerical results in Appendix~\ref{app:cost}. An important question to be answered is the impacts of $t$, which is affected by both $n$ and $d$. The maximum number of terms $t$ in $f$ is bounded by $\binom{n+d-1}{d}$. When both $n$ and $t$ increase, the value of $t$ grows exponentially. However, if either $n$ or~$d$ is fixed, the growth becomes polynomial.

We investigated the growth of $t$ values using a homogeneous polynomial $f$, where
\[
    f\left(x_1,x_2,\ldots,x_n\right) = \left(x_1+x_2+\cdots+x_n\right)^d.
\]
and~$d \in \Z^{+}$. The values of $t$ for various $n$ and $d$ values are shown in Tables~\ref{tab:stat_n_t} and~\ref{tab:stat_d_t}, respectively. Table~\ref{tab:stat_d_t} shows that $t$ increases when $d$ is fixed and $n$ increases. While the convergence rates increase for larger $d$, the growth is still polynomial. Similarly, $t$ increases when $n$ is fixed and $d$ increases in Table~\ref{tab:stat_n_t}. 

	\begin{center}
		\begin{table}%
		\centering
		\caption{number of terms $t$ versus dimension $n$ with fixed degree $d$\label{tab:stat_n_t}}
		\begin{tabular*}{350pt}{@{\extracolsep\fill}cccccccccccc@{\extracolsep\fill}}
		\toprule
		\textbf{$d \backslash n$} & 2 & 3 & 4 & 5 & 6 & 7 & 8 & 9 & 10 & \textbf{Convergence} \\
		\midrule
		3 & 4 & 10 & 20 & 35 & 56 & 84 & 120 & 165 & 220 & 2.50 \\
		4 & 5 & 15 & 35 & 70 & 126 & 210 & 330 & 495 & 715 & 3.10 \\
		5 & 6 & 21 & 56 & 126 & 252 & 462 & 792 & 1287 & 2002 & 3.64 \\
		6 & 7 & 28 & 84 & 210 & 462 & 924 & 1716 & 3003 & 5005 & 4.12 \\
		\bottomrule
		\end{tabular*}
		\end{table}
	\end{center}

\begin{center}
		\begin{table}%
		\centering
		\caption{number of terms $t$ versus degree $d$ with fixed dimension $n$\label{tab:stat_d_t}}
		\begin{tabular*}{350pt}{@{\extracolsep\fill}ccccccccccc@{\extracolsep\fill}}
		\toprule
		\textbf{$n\backslash d$} & 2 & 3 & 4 & 5 & 6 & 7 & 8 & 9 & 10 & \textbf{Convergence} \\
		\midrule
		2 & 3 & 4 & 5 & 6 & 7 & 8 & 9 & 10 & 11 & 0.81 \\
		3 & 6 & 10 & 15 & 21 & 28 & 36 & 45 & 55 & 66 & 1.50 \\
		4 & 10 & 20 & 35 & 56 & 84 & 120 & 165 & 220 & 286 & 2.10 \\
		5 & 15 & 35 & 70 & 126 & 210 & 330 & 495 & 715 & 1001 & 2.64 \\
		\bottomrule
		\end{tabular*}
		\end{table}
\end{center}

Actual~$t$ values are often smaller than the theoretical upper bounds as the coefficients of many terms may be zero. Operations like computing Hessian matrix~$H$ for $f$ may eliminate more terms, further reducing $t$, especially for small $d$. Moreover, routines that depend on $t$ are highly parallelizable. Therefore, the growth of $t$ is polynomial when either $n$ or $d$ is fixed and it will not jeopardize the performance of the diagonalization algorithm.

\bibliographystyle{plain}
\bibliography{diagonalization_complexity}



\end{document}